\documentclass[12pt,final,reqno]{amsart}
\usepackage[english]{babel}
\usepackage{amsmath,amsthm}
\usepackage{amssymb}
\usepackage{amsfonts}
\usepackage{graphics}
\usepackage{inputenc}
\usepackage{float}
\usepackage[left=2.5cm,right=2.5cm,top=2.3cm,bottom=2.3cm]{geometry}
\usepackage{mathrsfs}
\usepackage{color}
\usepackage{graphicx}%
\usepackage{color}
\usepackage{cite}
\usepackage[active]{srcltx}
\usepackage{float}

\usepackage{color}
\usepackage[unicode]{hyperref}
\hypersetup{
    colorlinks = true,%
    citecolor = [rgb]{0.0,0.0,0.9},
    filecolor=black,%
    linkcolor = [rgb]{0.65,0.0,0.0},%
    anchorcolor = red,
    pagecolor = red,
    urlcolor= [rgb]{0.65,0.0,0.0}
}

\usepackage{mathrsfs}
\newtheorem{proposition}{Proposition}[section]
\newtheorem{theorem}{Theorem}[section]
\newtheorem{lemma}{Lemma}[section]

\newtheorem{remark}{Remark}[section]

\numberwithin{equation}{section}
\newcommand{\N}{{\mathbb N}}
\newcommand{\R}{{\mathbb R}}

\newcommand{\eps}{{\varepsilon}}
\newcommand{\ds}{\displaystyle}

\numberwithin{equation}{section}

\newcommand{\Pl}{(\mathcal{SM}_\lambda)}
\newcommand{\kf}{\chi_{W^b}}
\newcommand{\kfG}{\chi_{W^b_{G}}}

\numberwithin{equation}{section}

\makeatletter
\renewcommand*{\@fnsymbol}[1]{\ensuremath{\ifcase#1\or *\or **\or \ddagger\or
   \mathsection\or \mathparagraph\or \|\or **\or \dagger\dagger
   \or \ddagger\ddagger \else\@ctrerr\fi}}

\g@addto@macro{\endabstract}{\@setabstract}
\newcommand{\authorfootnotes}{\renewcommand\thefootnote{\@fnsymbol\c@footnote}}%
\makeatother
\def\d{{{\rm d}%
}v_{g}}

\begin{document}

\begin{center}
  \large
\textbf{\textsc{Schr\"odinger-Maxwell systems on Hadamard
manifolds}}\par \bigskip

  \normalsize
  \authorfootnotes
 Csaba Farkas\footnote{Email address of Cs. Farkas: farkas.csaba2008@gmail.com}\textsuperscript{,1,3}
 Alexandru Krist\'aly \footnote{Email address of A. Krist\'aly: alexandrukristaly@yahoo.com}\textsuperscript{,2,3}
  \par \bigskip
\textsuperscript{1}{\footnotesize Department of Mathematics and Informatics, Sapientia University, Tg. Mure\c s, Romania}\\
\textsuperscript{2} {\footnotesize Department of Economics, Babe\c s-Bolyai University, Cluj-Napoca, Romania} \\
\textsuperscript{3}{\footnotesize  Institute of Applied Mathematics,
\'Obuda University, 1034 Budapest, Hungary}

 \par \bigskip

  \today
\end{center}
\hrulefill \\
{\bf Abstract.} {\footnotesize In this paper we study nonlinear
Schr\"odinger-Maxwell systems on $n-$dimensional Hadamard manifolds,
$3\leq n\leq 5.$ The main difficulty resides in the lack of
compactness of such manifolds which is recovered by exploring
suitable isometric actions. By combining variational arguments, some
existence, uniqueness and multiplicity of isometry-invariant weak
solutions are established for the Schr\"odinger-Maxwell system
depending on the behavior of the nonlinear term.}

\noindent \hrulefill


\section{Introduction and main results}\label{section1}

\subsection{Motivation}

The Schr\"odinger-Maxwell system
\begin{equation}\label{eredeti}
\ \left\{
\begin{array}{lll}
-\frac{\hbar^2}{2m}\Delta u+\omega u+e u\phi=f(x,u) & \mbox{in} &
\mathbb{R}^3 , \\
-\Delta\phi=4\pi e u^2 & \mbox{in} & \mathbb{R}^3,
\end{array}%
\right.
\end{equation}
 describes the statical behavior  of a charged non-relativistic
quantum mechanical particle interacting with the electromagnetic
field. More precisely, the unknown terms $u : \mathbb{R}^3
\to\mathbb{R}$ and $\phi : \mathbb{R}^3 \to \mathbb{R}$ are the
fields associated to the particle and the electric potential,
respectively. Here and in the sequel, the quantities $m$, $e$,
$\omega$ and $\hbar$ are the mass, charge, phase, and Planck's
constant, respectively, while $f:\mathbb R^3\times \mathbb R\to
\mathbb R$ is a Carath\'eodory function verifying some growth
conditions.  In fact, system (\ref{eredeti}) comes from the
evolutionary nonlinear Schr\"odinger equation by using a
Lyapunov-Schmidt reduction.

 The Schr\"odinger-Maxwell system (or its
variants) has been the object of various investigations in the last
two decades. Without sake of completeness, we recall in the sequel
some important contributions to the study of system (\ref{eredeti}).
Benci and Fortunato \cite{BF1} considered the case of
$f(x,s)=|s|^{p-2}s$ with $p\in (4,6)$ by proving the existence of
infinitely many radial solutions for (\ref{eredeti}); their main
step relies on the reduction of system (\ref{eredeti}) to the
investigation of critical points of a "one-variable" energy
functional associated with (\ref{eredeti}). Based on the idea of
Benci and Fortunato, under various growth assumptions on $f$ further
existence/multiplicity results can be found in Ambrosetti and Ruiz
\cite{Ambrosetti-Ruiz}, Azzolini
 \cite{azzolini1}, Azzollini, d'Avenia and Pomponio \cite{azzolini2}, d'Avenia \cite{davenia},
 d'Aprile and Mugnai \cite{mugnai}, Cerami and Vaira \cite{cerami}, Krist\'aly and Repovs \cite{kristalyrepovs}, Ruiz
\cite{ruiz}, Sun, Chen and Nieto \cite{SCN}, Wang and Zhou
\cite{wang}, Zhao and Zhao \cite{Zhao-Zhao}, and references therein.
By means of a Pohozaev-type identity, d'Aprile and Mugnai
\cite{mugnai2} proved the non-existence of non-trivial solutions to
system (\ref{eredeti}) whenever $f\equiv 0$ or $f(x,s)=|s|^{p-2}s$
and $p\in (0,2]\cup [6,\infty)$.

In recent years considerable efforts have been done to describe
various nonlinear phenomena in {\it curves spaces} (which are mainly
understood in linear structures), e.g.  optimal mass transportation
on metric measure spaces, geometric functional inequalities and
optimization problems on Riemannian/Finsler manifolds, etc. In
particular, this research stream  reached as well the study of
Schr\"odinger-Maxwell systems. Indeed, in the last five years
Schr\"odinger-Maxwell systems has been studied on $n-$dimensional
{\it compact Rieman\-nian mani\-folds} ($2\leq n \leq 5$) by Druet
and Hebey \cite{DH}, Hebey and Wei \cite{HW}, Ghimenti and
Micheletti \cite{GHM, GHM2} and Thizy \cite{Thizy-1, Thizy-2}. More
precisely, in the aforementioned papers various forms of the system
\begin{equation}\label{eredeti-2}
\ \left\{
\begin{array}{lll}
-\frac{\hbar^2}{2m}\Delta u+\omega u+e u\phi=f(u) & \mbox{in} &
M , \\
-\Delta_g\phi+\phi=4\pi e u^2 & \mbox{in} & M,
\end{array}%
\right.
\end{equation}
has been considered, where $(M,g)$ is a compact Riemannian manifold
 and $\Delta_g$ is the Laplace-Beltrami
operator, by proving existence results with further qualitative
property of the solution(s). As expected, the compactness  of
$(M,g)$ played a crucial role in these investigations.

As far as we know, no result is available in the literature
concerning Maxwell-Schr\"odinger systems on {\it non-compact
Riemannian manifolds}. Motivated by this fact, the purpose of the
present paper is to provide existence, uniqueness and multiplicity
results in the case of the Maxwell-Schr\"odinger system  in such a
non-compact setting. Since this problem is very general, we shall
restrict our study to {\it Hadamard manifolds} (simply connected,
complete Riemannian manifolds with non-positive sectional
curvature). Although any Hadamard manifold $(M,g)$ is diffeomorphic
to $\mathbb R^n$, $n={\rm dim}M$ (cf. Cartan's theorem), this is a
wide class of non-compact Riemannian manifold including important
geometric objects (as Euclidean spaces, hyperbolic spaces, the space
of symmetric positive definite matrices endowed with a suitable
Killing metric), see Bridson and Haefliger \cite{BH}.

To be more precise, we shall consider the Schr\"odinger-Maxwell
system
 \begin{equation*} \ \left\{
\begin{array}{lll}
-\Delta_g u+u+e u\phi=\lambda\alpha(x)f(u) & \mbox{in} &
M , \\
-\Delta_g \phi+\phi=qu^2 & \mbox{in} & M,
\end{array}%
\right. \eqno{(\mathcal{SM}_{\lambda})}
\end{equation*}
where $(M,g)$ is an $n-$dimensional Hadamard manifold $(3\leq n\leq
5)$,  $e,q>0$ are positive numbers, $f:\mathbb R\to \mathbb R$ is a
continuous function, $\alpha:M\to \mathbb R$ is a measurable
function,  and $\lambda>0$ is a parameter. The solutions $(u,\phi)$
of $(\mathcal{SM}_{\lambda})$ are sought in the Sobolev space
$H_g^1(M)\times H_g^1(M)$. In order to handle the lack of
compactness of $(M,g)$, a Lions-type symmetrization argument will be
used, based on the action of a suitable subgroup of the group of
isometries of $(M,g)$. More precisely, we shall adapt the main
results of Skrzypczak and Tintarev \cite{S-Tintarev} to our setting
concerning Sobolev spaces in the presence of group-symmetries. By
exploring variational arguments (principle of symmetric criticality,
minimization and mountain pass arguments), we consider the following
 problems describing roughly as well the main achievements:

\begin{itemize}
  \item[A.] {\it Schr\"odinger-Maxwell systems of Poisson type}:
  $\lambda=1$ and $f\equiv 1$. We prove the existence of the unique  weak
  solution $(u,\phi)\in H_g^1(M)\times H_g^1(M)$ to $(\mathcal{SM}_{1})$ while if $\alpha$ has some
  radial property (formulated in terms of the isometry group), the
  unique weak solution is isometry-invariant, see Theorem \ref{uniq-elso}.
  Moreover, we prove a rigidity result which states that
a specific profile function uniquely determines the structure of the
Hadamard manifold $(M,g)$, see Theorem \ref{rigidity}.
  \item[B.] {\it Schr\"odinger-Maxwell systems involving
sublinear terms at infinity}: $f$ is sublinear at infinity. We prove
that for small values of $\lambda>0$ system
$(\mathcal{SM}_{\lambda})$ has only the trivial solution, while for
enough large $\lambda>0$ the system $(\mathcal{SM}_{\lambda})$ has
at least two distinct, non-zero, isometry-invariant weak solutions,
see Theorem \ref{theorem-sublinear}.
  \item[C.] {\it Schr\"odinger-Maxwell systems involving oscillatory
  terms}: $f$ oscillates near the origin. We prove that system $(\mathcal{SM}_{1})$ has
  infinitely many distinct, non-zero,
isometry-invariant weak solutions which converge to 0 in the
$H^1_g(M)-$norm, see Theorem \ref{oszczero}.
\end{itemize}

\noindent In the sequel, we shall formulate rigourously our main
results with some comments.

\subsection{Statement of main results}

Let $(M,g)$ be an $n-$dimensional Hadamard manifold, $3\leq n \leq
6$. The pair $(u,\phi)\in H_g^1(M)\times H^1_g(M)$ is a {\it  weak
solution} to the system $(\mathcal{SM}_\lambda)$ if
\begin{equation}\label{schrodingerweak}
\int_M (\langle\nabla_g u , \nabla_g v\rangle+uv+e u \phi v )\d=
\lambda\int_M \alpha(x)f(u)v\d \hbox{ for all } v\in H^1_g(M),
\end{equation}
\begin{equation}\label{maxwellweak}
\int_M (\langle\nabla_g \phi, \nabla_g \psi\rangle +\phi \psi
)\d=q\int_M u^2 \psi\d \hbox{ for all }\psi \in H^1_g(M).
\end{equation}
For later use, we denote by ${\rm Isom}_g(M)$ the group of
isometries of $(M,g)$ and let $G$ be a subgroup of ${\rm
Isom}_g(M)$. A function $u:M\to \mathbb R$ is $G-${\it invariant} if
$u(\sigma(x))=u(x)$ for every $x\in M$ and $\sigma\in G$.
Furthermore, $u:M\to \mathbb R$ is \textit{radially symmetric
w.r.t.} $x_0\in M$ if $u$ depends on $d_g(x_0,\cdot)$, $d_g$ being
the Riemannian distance function. The {\it fixed point set of}  $G$
on $M$ is given by ${\rm
    Fix}_M(G)=\{x\in M:\sigma (x)=x\ {\rm for\ all}\ \sigma\in G\}.$ For a given $x_0\in M$, we introduce the following hypothesis which will be crucial in our investigations:\\

 \noindent    ${\boldsymbol{\ds{(H_G^{x_0})}}}$ {\it The group $G$ is a compact connected subgroup of ${\rm
        Isom}_g(M)$ such that ${\rm Fix}_M(G)=\{x_0\}$.}\\

\begin{remark}\label{fontosrm}\rm
In the sequel, we provide some concrete Hadamard manifolds and group
of isometries for which hypothesis ${\boldsymbol{\ds{(H_G^{x_0})}}}$
is satisfied:
\begin{itemize}
  \item {\it Euclidean spaces.} If $(M,g)=(\mathbb R^n,g_{\rm euc})$ is the usual Euclidean space, then $x_0=0$ and
$G={SO}(n_1)\times...\times {SO}(n_l)$ with $n_j\geq 2$, $j=1,...,l$
and $n_1+...+n_l=n$,
 satisfy ${\boldsymbol{\ds{(H_G^{x_0})}}}$, where
  $SO(k)$ is the special orthogonal group in dimension $k$.   Indeed, we have ${\rm
        Fix}_{\mathbb R^n}(G)=\{0\}$.
  \item {\it Hyperbolic spaces.} Let us consider the
    Poincar\'e ball model $\mathbb H^n=\{x\in \mathbb R^n:|x|<1\}$
    endowed with the Riemannian metric $\ds g_{\rm
        hyp}(x)=(g_{ij}(x))_{i,j={1,...,n}}=\frac{4}{(1-|x|^2)^2}\delta_{ij}$.
    It is well known  that $(\mathbb H^n,g_{\rm hyp})$ is a homogeneous Hadamard manifold with constant sectional curvature $-1$. Hypothesis ${\boldsymbol{\ds{(H_G^{x_0})}}}$ is verified with the same choices as above.

  \item {\rm {\it Symmetric positive definite matrices.}}
\rm Let  ${\rm Sym}(n,\mathbb R)$ be the set of symmetric $n\times
n$ matrices with real values, ${\rm P}(n,\mathbb R)\subset {\rm
Sym}(n,\mathbb R)$ be the cone of symmetric positive definite
matrices, and ${\rm P}(n,\mathbb R)_1$ be the subspace of matrices
in ${\rm P}(n,\mathbb R)$ with determinant one.
 The set ${\rm P}(n,\mathbb R)$ is endowed with the scalar product
$$
 \ \ \ \ \ \ \ \ \ \ \  \langle U,V \rangle_X={\rm Tr}(X^{-1}VX^{-1}U)\ \ {\rm for\ all}\ \ X\in
{\rm P}(n,\mathbb R),\ U,V\in T_X({\rm P}(n,\mathbb R))\simeq {\rm Sym}(n,\mathbb R), $$
where ${\rm Tr}(Y)$ denotes the trace of $Y\in {\rm Sym}(n,\mathbb
R)$. One can prove that $({\rm P}(n,\mathbb R)_1,\langle\cdot,\cdot
\rangle)$ is a homogeneous Hadamard manifold (with non-constant
sectional curvature) and the special linear group $SL(n)$ leaves
${\rm P}(n,\mathbb R)_1$ invariant and acts transitively on it.
Moreover, for every $\sigma\in {SL}(n )$, the map
 $[\sigma]:{\rm P}(n,\mathbb
R)_1\to {\rm P}(n,\mathbb R)_1$ defined   by $[\sigma](X)= \sigma
X\sigma^t$, is an isometry, where $\sigma^t$ denotes the
 transpose of $\sigma.$ 
If  $G={SO}(n)$, we can prove that ${\rm Fix}_{{\rm P}(n,\mathbb
R)_1}(G)=\{I_n\}$, where $I_n$ is the identity
 matrix; for more details, see  Krist\'aly \cite{kristaly}.
\end{itemize}
\end{remark}

\noindent For $x_0\in M$ fixed, we also introduce the hypothesis\\

 \noindent    ${\boldsymbol{\ds{(\alpha^{x_0})}}}$ {\it The function  $\alpha:M\to \mathbb R$ is non-zero, non-negative and radially symmetric  w.r.t. $x_0$}.\\

\noindent Our results are divided into three classes:\\

\noindent\textbf{ A. Schr\"odinger-Maxwell systems of Poisson type.}
Dealing with a Poisson-type system, we set  $\lambda=1$ and $f\equiv
1$ in $(\mathcal{SM}_\lambda)$. For abbreviation, we simply denote
$(\mathcal{SM}_1)$ by $(\mathcal{SM})$.
    \begin{theorem}\label{uniq-elso} Let $(M,g)$ be an $n-$dimensional  homogeneous  Hadamard
    manifold $(3\leq n\leq 6)$, and $\alpha\in L^2(M)$  be a non-negative function.
    Then there exists a unique, non-negative  weak solution $(u_0,\phi_0)\in H_g^1(M)\times
H^1_g(M)$ to problem  $(\mathcal{SM})$. Moreover, if
    $x_0\in M$ is fixed and $\alpha$ satisfies ${\boldsymbol{\ds{(\alpha^{x_0})}}}$, then  $(u_0,\phi_0)$ is $G-$invariant w.r.t.
any group $G\subset {\rm Isom}_g(M)$ which satisfies
${\boldsymbol{\ds{(H_G^{x_0})}}}$.
    \end{theorem}

\begin{remark}\rm Let $(M,g)$ be either the $n-$dimensional Euclidean space $(\mathbb R^n,g_{\rm euc})$ or
hyperbolic space $(\mathbb H^n,g_{\rm
        hyp})$, and fix $G={SO}(n_1)\times...\times {SO}(n_l)$ for a  splitting of $n=n_1+...+n_l$ with $n_j\geq 2$,
        $j=1,...,l$. If $\alpha$ is radially symmetric (w.r.t. $x_0=0$),  Theorem \ref{uniq-elso} states that the unique
        solution $(u_0,\phi_0)$ to the Poisson-type Schr\"odinger-Maxwell system
        $(\mathcal{SM})$ is not only invariant w.r.t. the group $G$
        but also with any compact connected subgroup $\tilde G$ of ${\rm Isom}_g(M)$ with the same fixed point property ${\rm Fix}_M(\tilde G)=\{0\}$;
        thus, in particular, $(u_0,\phi_0)$ is invariant w.r.t. the whole group
        $SO(n)$, i.e. $(u_0,\phi_0)$ is radially
        symmetric.
\end{remark}

\noindent For every $c\leq 0$, let ${\bf s}_{c},{\bf
ct}_{c}:[0,\infty)\to \mathbb
 R$ be defined by
\begin{equation}\label{kell}
    {\bf s}_{c}(r)=\left\{
    \begin{array}{lll}
    r
    & \hbox{if} &  {c}=0, \\
    \frac{\sinh(\sqrt{-c}r)}{\sqrt{-c}} & \hbox{if} & {c}<0,
    \end{array}\right.\ \ \ {\rm and}\ \ \ {\bf ct}_{c}(r)=\left\{
  \begin{array}{lll}
    \frac{1}{r}
    & \hbox{if} &  {c}=0, \\

  \sqrt{-c}\coth(\sqrt{-c}r) & \hbox{if} & {c}<0.
  \end{array}\right.
\end{equation} For $c\leq 0$  and  $3\leq n\leq 6$ we consider  the ordinary differential equations system
    \begin{equation*}
    \left\{
    \begin{array}
    [l]{l}%
    -h_{1}^{\prime\prime}(r)-(n-1)\boldsymbol{\rm ct_{c}}(s%
    )h_{1}^{\prime}(r)+h_{1}(r)+e h_{1}(r)h_{2}(r)-\alpha_0(r)=0,\ r\geq 0;\\
    -h_{2}^{\prime\prime}(r)-(n-1)\boldsymbol{\rm ct_{c}}(r%
    )h_{2}^{\prime}(r)+h_{2}(r)-q h_{1}(r)^{2}=0,\ r\geq 0;  \\
    \displaystyle\int_0^\infty (h_1'(r)^2+h_1^2(r)){\bf s}_{c}(r)^{n-1}{\rm d}r<\infty; \\ \displaystyle\int_0^\infty (h_2'(r)^2+h_2^2(r)){\bf s}_{c}(r)^{n-1}{\rm
    d}r<\infty,
    \end{array}
    \right.  \eqno{(\mathscr{R})}%
    \end{equation*}
   where $\alpha_0:[0,\infty)\to [0,\infty)$ satisfies the
   integrability condition $\alpha_0\in L^2([0,\infty),{\bf
   s}_{c}(r)^{n-1}{\rm d}r)$.

 We shall show (see Lemma \ref{uniqesolODE}) that $(\mathscr{R})$ has a unique,
    non-negative   solution   $(h_1^{c},h_2^{c})\in
    C^\infty(0,\infty)\times C^\infty(0,\infty)$. In fact, the following rigidity result can be stated:
    \begin{theorem}\label{rigidity}Let $(M,g)$ be an $n-$dimensional homogeneous Hadamard manifold  $(3\leq n\leq 6)$ with sectional curvature $\mathbf{K}\leq c\leq 0$. Let  $x_0\in M$ be fixed, and $G\subset {\rm Isom}_g(M)$ and
    $\alpha\in L^2(M)$  be
      such that hypotheses   ${\boldsymbol{\ds{(H_G^{x_0})}}}$ and ${\boldsymbol{\ds{(\alpha^{x_0})}}}$ are satisfied. If  $\alpha^{-1}(t)\subset M$ has null Riemannian measure for every $t\geq 0$,
      then the following statements are equivalent:
    \begin{itemize}
      \item[{\rm (i)}] $(h_1 ^{c}(d_{g}(x_{0},\cdot)),h_2 ^{c}(d_{g}(x_{0},\cdot))$ is the unique pointwise  solution of $({\mathcal{SM})};$
      \item[{\rm (ii)}] $(M,g)$ is isometric to the space form with constant sectional curvature ${\bf
      K}=c$. 
    \end{itemize}

    \end{theorem}

\noindent   \textbf{ B.  Schr\"odinger-Maxwell systems involving
sublinear terms at infinity.} In this part, we focus our attention
to Schr\"odinger-Maxwell systems involving sublinear nonlinearities.
To state our  result we consider a continuous function
$f:[0,\infty)\to \mathbb R$ which verifies the following
assumptions:
    \begin{itemize}
        \item[$(f_1)$] \label{f1-felt} $\frac{f(s)}{s}\to 0$ as $s\to 0^+$;
        \item[$(f_2)$] \label{f2-felt}   $\frac{f(s)}{s}\to 0$ as $s\to  \infty$;
        \item[$(f_3)$] \label{f3-felt} $F(s_0)>0$ for some $s_0>0$, where $\ds F(s)=\int_0^s f(t){\rm
        d}t,$ $s\geq  0.$
    \end{itemize}
\begin{remark} \rm \label{remark-zero-extension}

(a) Due to $(f_1)$, it is clear that $f(0)=0$, thus we can extend
continuously the function $f:[0,\infty)\to \mathbb R$ to
        the whole $\mathbb R$ by $f(s)=0$ for $s\leq 0;$ thus,  $F(s)=0$ for
        $s\leq 0$.

(b)  $(f_1)$ and $(f_2)$ mean that $f$ is superlinear at the origin
        and sublinear at infinity, respectively. The function
        $f(s)=\ln(1+s^2)$, $s\geq 0,$ verifies hypotheses $(f_1)-(f_3)$.

\end{remark}

    \begin{theorem}\label{theorem-sublinear}
Let $(M,g)$ be an $n-$dimensional homogeneous  Hadamard manifold
$(3\leq n\leq 5)$, $x_0\in M$ be fixed, and $G\subset {\rm
Isom}_g(M)$ and
    $\alpha\in L^1(M)\cap
L^\infty(M)$  be
      such that hypotheses   ${\boldsymbol{\ds{(H_G^{x_0})}}}$ and ${\boldsymbol{\ds{(\alpha^{x_0})}}}$ are satisfied.
 If the continuous function
$f:[0,\infty)\to \mathbb R$ satisfies
 assumptions $(f_1)-(f_3)$, then
        \begin{itemize}
            \item[{\rm (i)}]\label{elso-tetel-a} there exists  $\widetilde{\lambda}_0>0$ such that if $0<\lambda<\widetilde{\lambda}_0$, system  $(\mathcal{SM}_{\lambda})$ has only the
            trivial solution$;$
            \item[{\rm (ii)}]\label{elso-tetel-b}there exists  $\lambda_0>0$ such that for every $\lambda\geq \lambda_0$, system $(\mathcal{SM}_{\lambda})$ has at least two distinct non-zero,
            non-negative
            $G-$invariant weak solutions in $H_g^1(M)\times H^1_g(M)$.
        \end{itemize}
    \end{theorem}
\begin{remark}\rm
(a) By a three critical points result of Ricceri \cite{Ricceri} one
can prove
    that the number of solutions for system $(\mathcal{SM}_{\lambda})$ is stable under
    small nonlinear perturbations $g:[0,\infty)\to \mathbb R$ of
    subcritical type, i.e., $g(s)=o(|s|^{2^*-1})$ as $s\to \infty$,
    $2^*=\frac{2n}{n-2},$ whenever  $\lambda>\lambda_0$.

    (b) Working with sublinear nonlinearities, Theorem
\ref{theorem-sublinear} complements several results where $f$ has a
superlinear growth at infinity, e.g., $f(s)=|s|^{p-2}s$ with $p\in
(4,6)$.
\end{remark}


   \noindent \textbf{C. Schr\"odinger-Maxwell systems involving oscillatory terms.}
    Let $f:[0,\infty)\to \mathbb R$ be a continuous function with
$\ds F(s)=\int_0^s f(t)\mathrm{d}t$. We assume:
    \begin{itemize}
        \item[$(f_0^1)$]\label{f1-oszc} $\ds -\infty<\liminf_{s \to 0} \frac{F(s)}{s^2}\leq \limsup_{s\to 0}\frac{F(s)}{s^2}=+\infty$;
        \item[$(f_0^2)$]\label{f2-oszc} there exists a sequence $\{s_j\}_j \subset (0,1) $ converging to $0$ such that $f(s_j)<0$,  $j\in \N$.

    \end{itemize}
    \begin{theorem}\label{oszczero}
        Let $(M,g)$ be an $n-$dimensional homogeneous Hadamard manifold $(3\leq n\leq 5)$, $x_0\in M$ be fixed, and $G\subset {\rm Isom}_g(M)$ and
    $\alpha\in  L^1(M)\cap
L^\infty(M)$  be
      such that hypotheses   ${\boldsymbol{\ds{(H_G^{x_0})}}}$ and ${\boldsymbol{\ds{(\alpha^{x_0})}}}$ are satisfied. If $f:[0,\infty)\to \R$ is a continuous function
        satisfying $(f_0^1)$ and $(f_0^2)$, then there exists a sequence $\{(u_j^0,\phi_{u_j^0})\}_j\subset H^1_g(M)\times H^1_g(M)$ of distinct, non-negative $G-$invariant  weak solutions to $(\mathcal{SM})$ such that
        $$\lim_{j\to \infty}\|u_j^0\|_{H^1_g(M)}=\lim_{j\to \infty}\|\phi_{u_j^0}\|_{H^1_g(M)}=0.$$
    \end{theorem}

    \begin{remark}\rm (a) $(f_0^1)$ and $(f_0^2)$ imply $f(0)=0$; thus we can extend $f$ as in Remark \ref{remark-zero-extension} (a).

 (b) Under the assumptions of Theorem \ref{oszczero} we  consider  the  perturbed Schr\"odinger-Maxwell system
    \begin{equation*}
    \ \left\{
    \begin{array}{lll}
    -\Delta_g u+u+e u\phi=\lambda\alpha(x)[f(u)+\eps g(u)] & \mbox{in} &
    M , \\
    -\Delta_g \phi+\phi=qu^2 & \mbox{in} & M,
    \end{array}%
    \right. \eqno{(\mathcal{SM}_{\eps})}
    \end{equation*}
    where $\eps>0$ and $g:[0,\infty)\to \mathbb{R}$ is a continuous function with
    $g(0)=0$.
    Arguing as in the proof of Theorem \ref{oszczero}, a careful energy control provides the following statement: for every $k\in \mathbb{N}$ there exists $\eps_k>0$ such that $(\mathcal{SM}_{\eps})$ has at least $k$ distinct,
    $G-$invariant weak solutions $(u_{j,\eps},\phi_{u_{j,\eps}})$, $j\in \{1,...,k\}$, whenever $\eps \in [-\eps_k,\eps_k]$. Moreover, one can prove that
    $\|u_{j,\eps}\|_{H^1_g(M)}<\frac{1}{j}\ {\rm and}\ \|\phi_{u_{j,\eps}}\|_{H^1_g(M)}<\frac{1}{j},\ j\in \{1,...,k\}.$
   Note that a similar
    phenomenon has been described for Dirichlet problems in
    Krist\'aly and Moro\c sanu \cite{KM}.

    (c) Theorem \ref{oszczero} complements some results from the
    literature where  $f:\mathbb R\to \mathbb R$ has the symmetry
    property $f(s)=-f(-s)$ for every $s\in \mathbb R$ and verifies an Ambrosetti-Rabinowitz-type assumption. Indeed, in
    such cases, the symmetric version of the mountain pass theorem
    provides a sequence of weak solutions for the studied Schr\"odinger-Maxwell
    system.
    \end{remark}
\newpage

\section{Preliminaries}
\subsection{Elements from Riemannian geometry}
In the sequel, let $n\geq 3$ and $(M, g)$ be an $n-$dimensional
Hadamard manifold (i.e., $(M, g)$ is a complete, simply connected
Riemannian manifold with nonpositive sectional curvature). Let $T_x
M$ be the tangent space at $x \in M$, $\displaystyle TM =
\bigcup_{x\in M}T_xM$ be the tangent bundle, and $d_g : M \times M
\to [0, +\infty)$ be the distance function associated to the
Riemannian metric $g$. Let $B_g(x, \rho) = \{y \in M : d_g(x, y) <
\rho \}$ be the open metric ball with center $x$ and radius $\rho
> 0$. If $\d$ is the canonical volume element on $(M, g)$, the
volume of a bounded open set $S \subset M$ is $\mathrm{Vol}_g(S) =
\displaystyle\int_S \d$.
 If ${\text d}\sigma_g$ denotes
the $(n-1)-$dimensional Riemannian measure induced on $\partial S$
by $g$, $\mathrm{Area}_g(\partial S)=\displaystyle\int_{\partial S}
{\text d}\sigma_g$ denotes the area of $\partial S$ with respect to
the metric $g$.

Let $p>1.$ The norm of $L^p(M)$ is given by
$$\|u\|_{L^p(M)}=\left(\displaystyle\int_M
|u|^p\d\right)^{1/p}.$$ Let $u:M\to \mathbb R$ be a function of
class $C^1.$ If $(x^i)$ denotes the local coordinate system on a
coordinate neighbourhood of $x\in M$, and the local components of
the differential of $u$ are denoted by $u_i=\frac{\partial
u}{\partial
    x_i}$, then the local components of the gradient  $\nabla_g u$ are
$u^i=g^{ij}u_j$. Here, $g^{ij}$ are the local components of
$g^{-1}=(g_{ij})^{-1}$. In particular, for every $x_0\in M$ one has
the eikonal equation
\begin{equation}\label{dist-gradient}
|\nabla_g d_g(x_0,\cdot)|=1\ {\rm  \ on}\ M\setminus \{x_0\}.
\end{equation}
The Laplace-Beltrami operator is given by $\Delta_g u={\rm
div}(\nabla_g u)$ whose expression in a local chart of associated
coordinates $(x^i)$ is $$\Delta_g u=g^{ij}\left(\frac{\partial^2
u}{\partial x_i\partial x_j}-\Gamma_{ij}^k\frac{\partial u}{\partial
x_k}\right),$$ where $\Gamma_{ij}^k$ are the coefficients of the
Levi-Civita connection. For enough regular $f:[0,\infty)\to \mathbb
R$ one has the formula
\begin{equation}\label{Laplace-chain}
    -\Delta_g(f(d_g(x_0,x))=-f''(d_g(x_0,x))-f'(d_g(x_0,x))\Delta_g(d_g(x_0,x))
    \ {\rm for\ a.e.}\ x\in M.
\end{equation}
 When no
confusion arises, if $X,Y\in T_x M$, we simply write $|X|$ and
$\langle X,Y\rangle$ instead of the norm $|X|_x$ and inner product
$g_x(X,Y)=\langle X,Y\rangle_x$, respectively. The $L^p(M)$ norm of
$\nabla_g u(x)\in T_xM$ is given by
$$\|\nabla_g u\|_{L^p(M)}=\left(\displaystyle\int_M
|\nabla_gu|^p{\rm d}v_g\right)^\frac{1}{p}.$$  The space $H^1_g(M)$
is the completion of $C_0^\infty(M)$ w.r.t. the norm
$$\|u\|_{H^1_g(M)}=\sqrt{\|u\|_{L^2(M)}^2+\|\nabla_g u\|_{L^2(M)}^2}.$$
Since $(M,g)$ is an $n-$%
    di\-men\-sional  Hadamard manifold $(n\geq 3)$, according to Hoffman and
Spruck \cite{HS}, the embedding $H^1_g(M)\hookrightarrow L^p(M)$ is
continuous for every $p\in [2,2^*]$, where $2^*=\frac{2n}{n-2};$ see
also Hebey \cite{hebey}. Note that the embedding
$H^1_g(M)\hookrightarrow L^p(M)$ is not compact for any $p\in
[2,2^*]$.

For any $c\leq 0,$ let $\displaystyle
V_{c,n}(\rho)=n\omega_n\int_0^\rho {\bf s}_c(t)^{n-1}{\rm d}t$ be
the volume of the ball with radius $\rho>0$ in the $n-$dimensional
space form (i.e., either the hyperbolic space with sectional
curvature $c$ when $c<0$ or the Euclidean space when $c=0$), where
${\bf s}_c$ is from (\ref{kell}) and $\omega_n$ is the volume of the
unit $n-$dimensional Euclidean ball. Note that for every $x\in M$,
we have
\begin{equation}  \label{volume-comp-nullaban}
\lim_{\rho\to 0^+}\frac{\mathrm{Vol}_g(B_g(x,\rho))}{V_{c,n}(\rho)}%
=1.
\end{equation}
The notation ${\bf K}\leq {c}$ means that the sectional curvature is
bounded from above by ${c}$ at any point and  direction.

Let  $(M,g)$ be an $n-$%
    di\-men\-sional Hadamard manifold with sectional curvature ${\bf
        K}\leq c\leq 0$. Then we have (see
Shen \cite{Shen-volume} and Wu and Xin \cite[Theorems 6.1 \&
6.3]{Wu-Xin}):
\begin{itemize}
  \item {\it Bishop-Gromov volume comparison theorem}: the function $\rho\mapsto
    \frac{\mathrm{Vol}_g(B_g(x,\rho))}{V_{c,n}(\rho)}, \ \rho>0,$ is
    non-decreasing for every  $x\in M$. In particular, from
(\ref{volume-comp-nullaban}) we
    have
    \begin{equation}  \label{volume-comp-altalanos-0}
    {\mathrm{Vol}_g(B_g(x,\rho))}\geq V_{c,n}(\rho)\ {\rm for\ all}\ \rho>0.
    \end{equation}
    Moreover, if equality holds in (\ref{volume-comp-altalanos-0}) for all $x\in M$ and $\rho>0$ then $\mathbf{K}= c$.
  \item {\it Laplace comparison theorem}: $\Delta_gd_g(x_0,x)\geq (n-1){\bf
  ct}_c(d_g(x_0,x))$  for every  $x\in M\setminus \{x_0\}$.  If $\mathbf{K}= c$ then
  we have equality in the latter relation.
\end{itemize}

\subsection{Variational framework}
Let $(M,g)$ be an $n-$%
    di\-men\-sional Hadamard manifold, $3\leq n\leq 6.$ We define the energy functional $\mathscr{J}_\lambda:H^1_g(M)\times
H^1_g(M)\to \mathbb{R}$ associated with system $\Pl$, namely,
$$\mathscr{J}_\lambda(u,\phi)=\frac{1}{2}\|u\|_{H^1_g(M)}^2+\frac{e}{2}\int_M
\phi u^2 {\rm d}v_g-\frac{e}{4q}\int_M |\nabla_g \phi|^2{\rm
d}v_g-\frac{e}{4q}\int_M\phi^2 {\rm d}v_g-\lambda\int_M
\alpha(x)F(u) {\rm d}v_g.$$ In all our cases (see problems {\bf A},
{\bf B} and {\bf C} above), the functional $\mathscr{J}_\lambda$ is
well-defined and of class $C^1$ on $H^1_g(M)\times H^1_g(M)$. To see
this, we have to consider the second and fifth terms from
$\mathscr{J}_\lambda$; the other terms trivially verify the required
properties. First, a comparison principle and suitable Sobolev
embeddings give that there exists $C>0$ such that for every
$(u,\phi)\in H^1_g(M)\times H^1_g(M)$,
$$0\leq \int_M \phi u^2\d\leq \left(\int_M
\phi^{2^*}\d\right)^{\frac{1}{2^*}}\left(\int_M
|u|^{\frac{4n}{n+2}}\d\right)^{1-\frac{1}{2^*}}\leq
C\|\phi\|_{H^1_g(M)}\|u\|^2_{H^1_g(M)}<\infty,$$ where we used
$3\leq n\leq 6.$ If $\mathcal F:H^1_g(M)\to \mathbb R$ is the
functional defined by $\ds\mathcal{F}(u)=\int_M \alpha(x)F(u)\d$, we
have:
\begin{itemize}
   \item Problem {\bf A}: $\alpha\in L^2(M)$ and $F(s)=s$, $s\in \mathbb R,$   thus   $|\mathcal{F}(u)|\leq
   \|\alpha\|_{L^2(M)}\|u\|_{L^2(M)}<+\infty$ for all $u\in H^1_g(M).$
   \item Problems {\bf B} and {\bf C}: the assumptions allow
   to consider generically that $f$ is subcritical, i.e., there
   exist $c>0$ and $p\in [2,2^*)$ such that $|f(s)|\leq c(|s|+|s|^{p-1})$ for every $s\in \mathbb
   R$. Since $\alpha\in L^\infty(M)$ in every case, we have that $|\mathcal{F}(u)|<+\infty$ for
   every $u\in H^1_g(M)$ and $\mathcal{F}$ is of class $C^1$ on $H^1_g(M).$
\end{itemize}

\noindent {\sc Step 1.} {\it The pair $(u,\phi)\in H^1_g(M)\times
H^1_g(M)$ is a weak solution of $\Pl$ if and only if $(u,\phi)$ is a
critical point of $\mathscr{J}_\lambda$.} Indeed, due to relations
(\ref{schrodingerweak}) and (\ref{maxwellweak}), the claim follows.

By exploring an idea of Benci and Fortunato \cite{BF1}, due to the
Lax-Milgram theorem (see e.g. Brezis \cite[Corollary 5.8]{Brezis}),
we introduce the map $\phi_u:H_g^1(M)\to H_g^1(M)$ by associating to
every $u\in H^1_g(M)$ the unique solution $\phi=\phi_u$ of the
Maxwell equation
$$-\Delta_g \phi+\phi=qu^2.$$  We recall some important properties of the function
$u\mapsto \phi_u$ which are straightforward adaptations of
\cite[Proposition 2.1]{kristalyrepovs} and \cite[Lemma 2.1]{ruiz} to
the Riemannian setting:
\begin{eqnarray}
\label{a-tulajdonsag} \ds\|\phi_u\|_{H^1_g(M)}^2=q\int_M \phi_u u^2 \d,\ \ \  \phi_u\geq 0; \\
\ds u\mapsto \int_{M}\phi_u u^2\d \hbox{ is convex;}\label{c-tulajdonsag} \\
\ds \int_M \left(u\phi_u-v\phi_v\right)(u-v)\d\geq 0 \hbox{ for all
}u,v\in H^1_g(M)\label{d-tulajdonsag}.
\end{eqnarray}
%
 The "one-variable" energy
functional $\mathcal{E}_{\lambda}:H_g^1(M)\to \mathbb{R}$ associated
with  system $(\mathcal{SM}_\lambda)$ is defined by
\begin{equation}\label{one-dim-energy}
\mathcal{E}_{\lambda}(u)=\frac{1}{2}\|u\|_{H^1_g(M)}^2+\frac{e}{4}\int_M
\phi_u u^2 {\rm d}v_g-\lambda\ds\mathcal{F}(u).
\end{equation}
By using standard variational arguments, one has:\\

\noindent {\sc Step 2.} {\it The pair $(u,\phi)\in H^1_g(M)\times
H^1_g(M)$ is a critical point of $\mathscr{J}_\lambda$ if and only
if $u$ is a critical point of $\mathcal{E}_{\lambda}$ and
$\phi=\phi_u$.} Moreover, we  have that
\begin{equation}
\label{derivalt} \mathcal{E}_{\lambda}'(u)(v)=\int_M(\langle\nabla_g
u,\nabla_g v\rangle+uv+e\phi_u uv)\d-\lambda\int_M
\alpha(x)f(u)v\d.\end{equation}

In the sequel, let $x_0\in M$ be fixed, and $G\subset {\rm
Isom}_g(M)$ and
    $\alpha\in  L^1(M)\cap
L^\infty(M)$  be
      such that hypotheses  ${\boldsymbol{\ds{(H_G^{x_0})}}}$ and ${\boldsymbol{\ds{(\alpha^{x_0})}}}$ are satisfied.
The action of $G$ on $H^{1}_g(M)$ is defined by
\begin{equation}\label{action-of-the-group}
(\sigma u)(x)=u(\sigma^{-1}(x)) \ \ {\rm for\ all}\ \sigma\in G,\
u\in H^{1}_g(M),\ x\in M,
\end{equation}
where $\sigma^{-1}:M\to M$ is the inverse of the isometry $\sigma$.
Let
$$H^{1}_{g,G}(M)=\{u\in H^{1}_g(M):\sigma u=u\
{\rm for\ all}\ \sigma\in G\}$$ be the subspace of $G-$invariant
functions of $H^{1}_g(M)$ and $\mathcal
E_{\lambda,G}:H^{1}_{g,G}(M)\to \mathbb R$ be the rest\-riction of
the energy functional $\mathcal E_{\lambda}$ to $H^{1}_{g,G}(M)$.
The
following statement is crucial in our investigation:\\

\noindent {\sc Step 3.} {\it If $u_G\in H^{1}_{g,G}(M)$ is a
critical point of $\mathcal E_{\lambda,G}$, then it is a critical
point also for $\mathcal E_{\lambda}$ and $\phi_{u_G}$ is
$G-$invariant.}


{\it Proof of Step 3.} For the first part of the proof, we follow
Krist\'aly \cite[Lemma 4.1]{kristaly}. Due to relation
(\ref{action-of-the-group}), the group $G$ acts continuously on
$H^1_g(M)$.

We claim that $\mathcal E_{\lambda}$ is $G-$invariant. To prove
this, let  $u\in H^1_g(M)$ and $\sigma\in G$ be fixed. Since
$\sigma:M\to M$ is an isometry on $M$,  we have by
(\ref{action-of-the-group}) and the chain rule that $\nabla_g(\sigma
u)(x)=D \sigma_{\sigma^{-1}(x)} \nabla_g u(\sigma^{-1}(x))$ for
every $x\in M$, where
$D\sigma_{\sigma^{-1}(x)}:T_{\sigma^{-1}(x)}M\to T_x M$ denotes the
differential of  $\sigma$ at the point $\sigma^{-1}(x)$. The
(signed) Jacobian determinant of $\sigma$ is 1 and
$D\sigma_{\sigma^{-1}(x)}$ preserves inner products; thus, by
relation (\ref{action-of-the-group}) and a change of variables
$y=\sigma^{-1}(x)$ it turns out that
    \begin{eqnarray*}
        \|\sigma u\|_{H^1_g(M)}^2 &=& \int_M \left(|\nabla_g (\sigma u)(x)|_x^2 + |(\sigma
        u)(x)|^2\right){\rm d}v_g(x) \\
        &=&\int_M \left(|\nabla_g
        u(\sigma^{-1}(x))|_{\sigma^{-1}(x)}^2 +
        |u(\sigma^{-1}(x))|^2\right){\rm d}v_g(x) \\ &=&\int_M \left(|\nabla_g
        u(y)|_y^2 + |u(y)|^2\right){\rm
        d}v_g(y)\\&=&\|u\|_{H^1_g(M)}^2.
    \end{eqnarray*}
According to ${\boldsymbol{\ds{(\alpha^{x_0})}}}$, one has that
$\alpha(x)=\alpha_0(d_g(x_0,x))$ for some function
$\alpha_0:[0,\infty)\to \mathbb R$. Since ${\rm Fix}_M(G)=\{x_0\}$,
we have for every $\sigma\in G$ and $x\in M$ that
$$\alpha(\sigma (x))=\alpha_0(d_g(x_0,\sigma (x)))=\alpha_0(d_g(\sigma (x_0),\sigma ( x)))=\alpha_0(d_g( x_0, x))=\alpha(x).$$
Therefore,
\begin{eqnarray*}
  \mathcal F(\sigma u) &=& \int_M \alpha(x)F((\sigma u)(x)){\rm d}v_g(x)=\int_M \alpha(x)F(u(\sigma^{-1} (x))){\rm d}v_g(x)
   = \int_M \alpha(y)F(u(y)){\rm d}v_g(y) \\
   &=&\mathcal F(u).
\end{eqnarray*}
We now consider the Maxwell equation
    $-\Delta_{g} \phi_{\sigma u}+\phi_{\sigma u} =q (\sigma u)^2$ which reads pointwisely as  $-\Delta_g \phi_{\sigma u}(y)+\phi_{\sigma u}(y) =q
    u(\sigma^{-1}(y))^2,$ $y\in M$.
After a change of variables one has $-\Delta_g \phi_{\sigma
u}(\sigma (x))+\phi_{\sigma u}(\sigma (x)) =q u(x)^2,$ $x\in M,$
which means by the uniqueness that $\phi_{\sigma u}(\sigma
(x))=\phi_u (x).$ Therefore,
\begin{align*}
\int_M \phi_{\sigma u}(x) (\sigma u)^2(x){\rm d}v_g(x)&=\int_M
\phi_{u}(\sigma^{-1} (x))u^2(\sigma^{-1} (x)){\rm d}v_g(x)
\overset{x=\sigma(y)}{=}\int_M \phi_{ u}(y) u^2( y) {\rm d}v_g(y),
\end{align*} which proves the $G-$invariance of $u\mapsto \displaystyle \int_M \phi_u u^2{\rm d}v_g$, thus the claim.

Since the fixed point set of $H^1_{g}(M)$ for  $G$ is precisely
$H^1_{g,G}(M)$, the principle of symmetric criticality of Palais
\cite{Palais} shows that every critical point $u_G\in H^1_{g,G}(M)$
of the functional $\mathcal E_{\lambda,G}$ is also a critical point
of $\mathcal E_\lambda$. Moreover, from the above uniqueness
argument, for every $\sigma\in G$ and $x\in M$ we have
$\phi_{u_G}(\sigma x)=\phi_{\sigma u_G}(\sigma x)=\phi_{u_G}(x)$, i.e., $\phi_{u_G}$ is $G-$invariant.  \hfill $\square$\\


\noindent Summing up \textsc{Steps 1-3}, we have the following
implications: for an element $u\in H^1_{g,G}(M)$,
\begin{equation}\label{implikaciok}
    \mathcal E_{\lambda,G}'(u)=0\   \Rightarrow\   \mathcal
E_{\lambda}'(u)=0\    \Leftrightarrow\   \mathscr
J_{\lambda}'(u,\phi_u)=0\   \Leftrightarrow (u,\phi_u)\  {\rm is\ a\
weak\ solution\ of}\ (\mathcal{SM}_\lambda).
\end{equation}
 Consequently, in order
to guarantee $G-$invariant weak solutions for
$(\mathcal{SM}_\lambda)$, it is enough to produce critical points
for the energy functional $\mathcal E_{\lambda,G}:H^1_{g,G}(M)\to
\mathbb R$. While the embedding $H^{1}_g(M)\hookrightarrow L^p(M)$
is only continuous for every $p\in [2,2^*]$, we adapt the main
results from Skrzypczak and Tintarev \cite{S-Tintarev} in order to
regain some compactness by exploring the presence of group
symmetries:
\begin{proposition}\label{Tintarev} {\rm \cite[Theorem 1.3 \& Proposition 3.1]{S-Tintarev}} Let $(M, g)$ be an  $n-$dimensional homogeneous Ha\-da\-mard
    manifold and $G$ be a compact connected subgroup of ${\rm
        Isom}_g(M)$ such that ${\rm Fix}_M(G)$ is a singleton. Then  $H^1_{g,G}(M)$ is compactly embedded into $L^p(M)$ for every $p\in (2,2^*)$.
\end{proposition}

\section{Proof of the main results}
\subsection{Schr\"odinger-Maxwell systems of Poisson type}

Consider the operator $\mathscr{L}$ on $H^1_{g}(M)$ given by
$$\mathscr{L}(u)=-\Delta_g u +u+e\phi_u u.$$ The following
comparison principle can be stated:
\begin{lemma}\label{comparison}
    Let $(M,g)$ be an $n-$dimensional Hadamard manifold $(3\leq n\leq
    6)$, $u,v\in H_g^1(M)$.
\begin{itemize}
  \item[{\rm (i)}] If $\mathscr{L}(u)\leq \mathscr{L}(v)$   then $u\leq v$.
  \item[{\rm (ii)}] If $0\leq u\leq v$ then $\phi_u\leq \phi_v$.
\end{itemize}
\end{lemma}

{\it Proof.} {\rm (i)}   Assume that $A=\{x\in M: u(x)>v(x)\}$ has a
    positive Riemannian measure. Then multiplying $%
    \mathscr{L}(u)\leq \mathscr{L}(v)$ by $(u-v)_+$, an integration yields that
    $$\int_A |\nabla_g u-\nabla_g v|^2\d+\int_A (u-v)^2\d+e\int_A (u \phi_u-v\phi_v) (u-v)\d \leq
    0.$$
The latter inequality and relation \eqref{d-tulajdonsag}
 produce a contradiction.

(ii) Assume that $B=\{x\in M: \phi_u(x)>\phi_v(x)\}$ has a
    positive Riemannian measure. Multiplying the Maxwell-type equation $-\Delta_g
    (\phi_u-\phi_v)+\phi_u-\phi_v=q(u^2-v^2)$ by
    $(\phi_u-\phi_v)_+$, we obtain that
  $$\int_B|\nabla_g \phi_u-\nabla_g \phi_v|^2\d+\int_B (\phi_u-\phi_v)^2\d
  =q\int_B(u^2-v^2)(\phi_u-\phi_v)\d\leq 0,
    $$
    a contradiction.
\hfill $\square$\\

{\it Proof of Theorem \ref{uniq-elso}.} Let $\lambda=1$ and for
simplicity, let $\mathcal{E}=\mathcal{E}_{1}$ be the energy
functional from (\ref{one-dim-energy}). First of all, the function
$\ds u\mapsto \frac{1}{2}\|u\|_{H^1_g(M)}^2$ is strictly convex on
$H^1_{g}(M).$ Moreover, the linearity of $u\mapsto \mathcal
F(u)=\displaystyle \int_M \alpha(x)u(x){\rm d}v_g(x)$ and property
\eqref{c-tulajdonsag} imply that the energy functional $\mathcal{E}$
is strictly convex on $H^1_{g}(M)$. Thus $\mathcal{E}$ is
sequentially weakly lower semicontinuous on $H^1_{g}(M)$,  it is
bounded from below and coercive. Now the basic result of the
calculus of variations implies that $\mathcal{E}$ has a unique
(global) minimum point $u\in H^1_{g}(M)$, see Zeidler \cite[Theorem
38.C and Proposition 38.15]{Zeidler}, which is also the unique
critical point of $\mathcal{E},$ thus $(u,\phi_u)$ is the unique
weak solution of $(\mathcal{SM})$.  Since $\alpha \geq 0,$ Lemma
\ref{comparison} (i) implies that $u\geq 0.$

Assume the function $\alpha$ satisfies
${\boldsymbol{\ds{(\alpha^{x_0})}}}$ for some $x_0\in M$  and let
$G\subset {\rm Isom}_g(M)$ be such that
${\boldsymbol{\ds{(H_G^{x_0})}}}$ holds. Then we can repeat the
above arguments for
 $\mathcal{E}_{1, G}=\mathcal{E}|_{H^1_{g,G}(M)}$ and $H^1_{g,G}(M)$ instead of $\mathcal
 E$  and $H^1_{g}(M)$, respectively, obtaining by (\ref{implikaciok}) that $(u,\phi_u)$ is a $G-$invariant  weak solution for $(\mathcal{SM})$.  \hfill $\square$\\

In the sequel we focus our attention to the system  $(\mathscr{R})$
from \S \ref{section1}; namely, we have
\begin{lemma}\label{uniqesolODE}
    System $(\mathscr{R})$ has a unique, non-negative pair of solutions belonging to $C^\infty(0,\infty)\times C^\infty(0,\infty)$.
\end{lemma}

{\it Proof.} Let $c\leq 0$ and $\alpha_0\in L^2([0,\infty),{\bf
   s}_{c}(r)^{n-1}{\rm d}r)$.
    Let us consider the Riemannian space form $(M_c,g_c)$
    with  constant sectional curvature $c\leq 0$, i.e., $(M_c,g_c)$ is
    either the Euclidean space $(\mathbb R^n,g_{\rm euc})$ when $c=0$, or the hyperbolic space $(\mathbb H^n,g_{\rm hyp})$ with (scaled) sectional curvature $c<0.$
    Let $x_0\in M$ be fixed and $\alpha(x)=\alpha_0(d_{g_c}(x_0,x)),$ $x\in M.$ Due to the integrability assumption on $\alpha_0$, we have that $\alpha\in L^2(M)$. Therefore, we are in the position to apply
    Theorem \ref{uniq-elso} on $(M_c,g_c)$  (see examples from Remark \ref{fontosrm}) to the problem
    \begin{equation*}
    \ \left\{
    \begin{array}{lll}
    -\Delta_g u+u+e u\phi=\alpha(x) & \mbox{in} &
    M_c , \\
    -\Delta_g \phi+\phi=qu^2 & \mbox{in} & M_c,
    \end{array}%
    \right. \eqno{(\mathcal{SM}_c)}
    \end{equation*}
    concluding that it has a unique, non-negative  weak solution $(u_0,\phi_{u_0})\in
H_{g_c}^1(M_c)\times H_{g_c}^1(M_c)$, where $u_0$ is the unique
global minimum point of the "one-variable" energy functional
associated with problem $({\mathcal SM}_{c})$. Since $\alpha$ is
radially symmetric in $M_c$, we may consider
    the group $G=SO(n)$ in the second part of Theorem \ref{uniq-elso} in order to prove that $(u_0,\phi_{u_0})$ is $SO(n)-$invariant, i.e., radially
    symmetric. In particular, we can represent these functions as
    $u_0(x)=h_1^c(d_{g_c}(x_0,x))$ and
    $\phi_0(x)=h_2^c(d_{g_c}(x_0,x))$ for some  $h_i^c:[0,\infty)\to
    [0,\infty)$, $i=1,2.$
By using formula (\ref{Laplace-chain}) and the Laplace comparison
theorem for ${\bf K}=c$ it follows that the equations from
$({\mathcal SM}_{c})$ are transformed into the first two equations
of $({\mathscr R})$ while the second two relations in $({\mathscr
R})$  are nothing but the "radial" integrability conditions
inherited from the fact that $(u_0,\phi_{u_0})\in
H_{g_c}^1(M_c)\times H_{g_c}^1(M_c).$ Thus, it turns out that
problem
    $({\mathscr R})$ has a non-negative pair of solutions  $(h_1^{c},h_2^{c})$. Standard regularity results show that $(h_1^{c},h_2^{c})\in
    C^\infty(0,\infty)\times C^\infty(0,\infty)$.
    Finally, let us assume that $({\mathscr R})$ has another non-negative pair of solutions $(\tilde h_1^{c},\tilde
    h_2^{c})$,  distinct from   $(h_1^{c},h_2^{c})$. Let $\tilde u_0(x)=\tilde h_1^c(d_{g_c}(x_0,x))$ and
    $\tilde \phi_0(x)=\tilde h_2^c(d_{g_c}(x_0,x))$.
There are two cases: (a) if $h_1^{c}= \tilde h_1^{c}$ then
$u_0=\tilde u_0 $ and by the uniqueness of solution for  the Maxwell
equation it follows that $\phi_0=\tilde \phi_0$, i.e., $h_2^{c}=
\tilde h_2^{c}$, a contradiction; (b) if $h_1^{c}\neq  \tilde
h_1^{c}$ then  $u_0\neq \tilde u_0 $. But the latter relation is
absurd since  both elements $u_0$ and $\tilde u_0 $ appear as unique
global minima of the "one-variable" energy functional associated
with $({\mathcal{SM}}_{c})$. \hfill
$\square$\\

{\it Proof of Theorem \ref{rigidity}.} "(ii)$\Rightarrow$(i)": it
follows directly from Lemma \ref{uniqesolODE}.

"(i)$\Rightarrow$(ii)": Let $x_0\in M$ be fixed and assume that the
pair $(h_1 ^{c}(d_{g}(x_{0},\cdot)),h_2 ^{c}(d_{g}(x_{0},\cdot))$ is
the unique pointwise solution to
    $({\mathcal{SM})}$, i.e.,
    $$\left\{
    \begin{array}{ll}
    -\Delta_g h_1^c(d_g(x_0,x))+h_1^c(d_g(x_0,x))+eh_1^c(d_g(x_0,x))h_2^c(d_g(x_0,x))=\alpha(d_g(x_0,x)), \ x\in M,\\
        -\Delta_g h_2^c(d_g(x_0,x))+h_2^c(d_g(x_0,x))=qh_1^c(d_g(x_0,x))^2, \ x\in M.
    \end{array}
    \right.$$
    By applying formula (\ref{Laplace-chain}) to the second equation,
    we  arrive to
     $$-h_{2}^c(d_g(x_0,x))''-h_{2}^c(d_g(x_0,x))'\Delta_g(d_g(x_0,x))+h^c_{2}(d_g(x_0,x))=q h_1^c(d_g(x_0,x))^2,\ x\in M.$$
Subtracting the second equation of the system $(\mathscr{R})$ from
the above one, we have that
\begin{equation}\label{h2egyenlet}
    h_{2}^c(d_g(x_0,x))'[\Delta_g(d_g(x_0,x))-(n-1)\textbf{ct}_c(d_g(x_0,x))]=0,
\ x\in M.
\end{equation}
 Let us suppose that there exists a set $A\subset M$ of
non-zero Riemannian measure such that $h_{2}^c(d_g(x_0,x))'=0$ for
every $x\in A.$ By a continuity reason, there exists a
non-degenerate interval $I\subset \mathbb R$ and a constant $c_0\geq
0$ such that $h_{2}^c(t)=c_0$ for every $t\in I$. Coming back to the
system $(\mathscr{R})$, we observe that
$h_{1}^c(t)=\sqrt{\frac{c_0}{q}}$ and
$\alpha_0(t)=\sqrt{\frac{c_0}{q}}(1+ec_0)$ for every $t\in I$.
Therefore,
$\alpha(x)=\alpha_0(d_g(x_0,x))=\sqrt{\frac{c_0}{q}}(1+ec_0)$ for
every $x\in A$, which contradicts the assumption on $\alpha.$

Consequently, by (\ref{h2egyenlet}) we have
    $\Delta_gd_g(x_0,x)= (n-1){\bf ct}_c(d_g(x_0,x))$ pointwisely on $M$. The latter
    relation can be equivalently transformed into
    \begin{equation}\label{delta-egyenloseg}\Delta_g%
    w_c(d_{g}(x_{0},x))=1,\ x\in M,$$ where
    $$w_c(r)=\int_{0}^{r}{\bf s}_c(s)^{-n+1}\int_{0}^{s}{\bf s}_c(t)%
    ^{n-1}{\rm d}t{\rm d}s.\end{equation} Let $0<\tau$ be fixed
    arbitrarily. The unit outward normal vector to the forward geodesic
    sphere $S_g(x_0,\tau)=\partial B_g(x_0,\tau)=\{x\in
    M:d_g(x_0,x)=\tau\}$ at $x\in S_g(x_0,\tau)$ is given by $\mathbf{n}=%
    \nabla_g d_g(x_0,x)$. Let us denote by ${\text
        d}\varsigma_g(x)$ the canonical volume form on $S_g(x_0,\tau)$
    induced by ${\rm d}v_g(x)$.
    By Stoke's formula and $%
    \langle\mathbf{n},\mathbf{n}\rangle=1$
    we have that
    \begin{eqnarray*}
        \mathrm{Vol}_{g}(B_g(x_{0},\tau )) &=&\int_{B_g(x_{0},\tau )}%
        \Delta_g(w_c(d_g(x_0,x)))\d =\int_{B_g(x_{0},\tau )}\mathrm{div}(\nabla_g
        (w_c(d_g(x_0,x))))\d \\
        &{=}&\int_{S_g(x_{0},\tau )}\langle\mathbf{n},w_c^{\prime }(d_g(x_0,x))%
        \nabla_g d_g(x_0,x\rangle \d \\
        &=&w_c^{\prime }(\tau ) \mathrm{Area}_{g}({S_g(x_{0},\tau
            )}).
    \end{eqnarray*}%
    Therefore,
    \begin{equation*}
    \frac{\mathrm{Area}_{g}({S_g(x_{0},\tau )})}{\mathrm{Vol}%
        _{g}(B_g(x_{0},\tau ))}=\frac{1}{w_c^{\prime }(\tau )}=\frac{{\bf s}_c%
        (\tau )^{n-1}}{\displaystyle\int_{0}^{\tau }{\bf s}_c%
        (t)^{n-1}{\rm d}t}.
    \end{equation*}%
Integrating the latter expression, it follows
    that
    \begin{equation}\label{egyenloseg-utolssso}
    \frac{\mathrm{Vol}_{g}(B_g(x_{0},\tau
        ))}{V_{c,n}(\tau)}=\lim_{s\to
        0^+}\frac{\mathrm{Vol}_{g}(B_g(x_{0},s ))}{V_{c,n}(s)}=1.
    \end{equation}
 In fact, the Bishop-Gromov volume comparison theorem implies that $$\frac{\mathrm{Vol}_{g}(B_g(x,\tau
        ))}{V_{c,n}(\tau)}=1 \hbox{ for all } x\in M,\ \tau>0.$$
    Now, the above equality implies that the sectional curvature is constant, ${\bf K}=c$, which concludes the proof.
   \hfill $\square$

\subsection{Schr\"odinger-Maxwell systems involving
sublinear terms at infinity} In this subsection we  prove Theorem
\ref{theorem-sublinear}.

(i) Let $\lambda\geq 0.$ If we choose $v=u$ in
\eqref{schrodingerweak} we obtain that
$$\int_M \left(|\nabla_g u|^2+u^2+e  \phi_u u^2 \right)\d=
\lambda\int_M \alpha(x)f(u)u\d.$$
    Due to the assumptions $(f_1)-(f_3)$, the number $\ds c_f=\max_{s>{0}}\frac{f(s)}{s}$ is well-defined and positive.
    Thus, by  \eqref{a-tulajdonsag} we have that
    $$\ds \|u\|_{H^1_g(M)}^2\leq \lambda c_f \|\alpha\|_{L^\infty(M)}\int_M u^2 \d\leq \lambda c_f \|\alpha\|_{L^\infty(M)}  \|u\|_{H^1_g(M)}^2.$$
    Therefore,  if $\lambda<c_f^{-1}\|\alpha\|_{L^\infty(M)}^{-1}:=\widetilde{\lambda}_0$, then the last inequality gives  $u=0$.
By the Maxwell equation we also have that $\phi=0$, which concludes
the proof of (i).

    (ii) The proof is divided into
    several steps.

{\sc Claim 1.} {\it The  energy functional $\mathcal E_\lambda$
 is coercive for every $\lambda\geq0$.} Indeed, due
to ({$f_2$}), we have that for every $\eps>0$ there exists
$\delta>0$ such that $|F(s)|\leq \eps |s|^2$  for every $
|s|>\delta.$ Thus
    \begin{align*}\mathcal F(u) &=\int_{\{u>\delta\}}\alpha(x)F(u)\d+\int_{\{u\leq \delta \}}\alpha(x)F(u)\d \leq \eps \|\alpha\|_{L^\infty(M)}  \|u\|_{H^1_g(M)}^2+\|\alpha\|_{L^1(M)} \max_{|s|\leq \delta}|F(s)|.
    \end{align*}
    Therefore (see \eqref{one-dim-energy}), $$
    \mathcal E_\lambda (u)\geq \left( \frac{1}{2}-\eps \lambda  \|\alpha\|_{L^\infty(M)} \right)\|u\|_{H^1_g(M)}^2-\lambda \|\alpha\|_{L^1(M)}  \sup_{|s|\leq \delta}|F(s)|.
    $$
    In particular, if $0<\varepsilon< (2 \lambda \| \alpha\|_{L^\infty(M)})^{-1}$ then $\mathcal{E}_{\lambda}(u) \to \infty$ as $\|u\|_{H^1_g(M)}\to \infty$.

{\sc Claim  2.} {\it $\mathcal{E}_{\lambda,G}$ satisfies the
Palais-Smale condition for  every $\lambda\geq0$.}
    Let $\{u_j\}_j\subset H^1_{g,G}(M)$ be a Palais-Smale sequence, i.e., $\{\mathcal{E}_{\lambda,G}(u_j)\}$ is bounded and
$\Vert(\mathcal{E}_{\lambda,G})^{\prime}(u_j)\Vert_{H^1_{g,G}(M)^{\ast}}\rightarrow0$
as $j\rightarrow\infty.$ Since $\mathcal{E}_{\lambda,G}$ is
coercive, the sequence $\{u_j\}_j$ is
    bounded in $H^1_{g,G}(M)$. Therefore, up to a subsequence, Proposition \ref{Tintarev} implies
    that $\{u_j\}_j$ converges  weakly in $H^1_{g,G}(M)$ and
    strongly in $L^{p}(M)$, $p\in (2,2^*),$ to an element $u\in
    H^1_{g,G}(M)$. Note that
    \[
    \int_{M}| \nabla_g u_j-\nabla_g
    u|^2 \d+\int_{M}\left(  u_j-u\right)^2  \d=
    \]%
    \[
    (\mathcal{E}_{\lambda,G})^{\prime}(u_j)(u_j-u)
    +(\mathcal{E}_{\lambda,G})^{\prime}(u)(u-u_j)+\lambda
    \int_{M}\alpha(x)[f(u_j(x))-f(u(x))](u_j-u)\d.
    \]
    Since $\Vert(\mathcal{E}_{\lambda,G})^{\prime}(u_j)\Vert_{H^1_{g,G}(M)^{\ast}}\rightarrow0$  and $u_j%
    \rightharpoonup u$ in $H^1_{g,G}(M)$, the first
    two terms at the right hand side tend to $0$.
Let $p\in (2,2^*).$ By the  assumptions, for every $\eps>0$ there
exists a constant $C_\eps>0$ such that $|f(s)|\leq \eps
|s|+C_\eps|s|^{p-1}$ for every $s\in \mathbb R$.
 The latter relation, H\"older inequality and  the fact that  $u_j\rightarrow u$
     in $L^p(M)$ imply  that
    \[
    \left\vert
    \int_{M}\alpha(x)[f(u_j)-f(u)](u_j-u)\d\right\vert\rightarrow0,
    \] as $j\to \infty.$
    Therefore, $\Vert u_j-u\Vert_{H^1_g(M)}^{2} \to 0$ as $j\to \infty.$

{\sc Claim 3.}  {\it $\mathcal{E}_{\lambda,G}$ is sequentially
weakly lower semicontinuous for every $\lambda\geq0$.}
 First, since $\|\cdot \|_{H^1_g(M)}^2$ is convex, it is also sequentially weakly lower
semicontinuous on $H^1_{g,G}(M)$. We shall prove that if $u_j
\rightharpoonup u$ in $H^1_{g,G}(M)$, then $\ds\int_M \phi_{u_j}
u_j^2 \d \to \int_M \phi_{u} u^2 \d.$ To see this, by Proposition
\ref{Tintarev} we have (up to a subsequence) that
    that $\{u_j\}_j$ converges  to $u$
    strongly in $L^{p}(M)$, $p\in (2,2^*).$
Let us consider the  Maxwell equations $-\Delta_g
\phi_{u_j}+\phi_{u_j}=qu_j^2$ and $-\Delta_g \phi_u+\phi_u=qu^2$.
Subtracting one from another and multiplying the expression by
$(\phi_{u_j}-\phi_u)$,  an integration and H\"older inequality yield
that
$$\|\phi_{u_j}-\phi_u\|_{H^1_g(M)}^2=q\int_M(u_j^2-u^2)(\phi_{u_j}-\phi_u){\rm
d}v_g\leq
C\|u_j-u\|_{L^{\frac{4n}{n+2}}(M)}\|u_j+u\|_{H^1_g(M)}\|\phi_{u_j}-\phi_u\|_{H^1_g(M)},$$
for some $C>0.$ Since $\frac{4n}{n+2}<2^*$ (note that $n\leq 5$),
the first term of the right hand side tends to $0$, thus we get that
$\phi_{u_j}\to \phi_u$ in $H_{g,G}^1(M)$ as $j\to \infty$. Now, the
desired limit follows from a H\"older inequality.

It remains to prove that $\mathcal{F}$
    is sequentially weakly continuous. To see this, let us suppose the contrary, i.e., let $\{u_j\}\subset H^1_{g,G}(M)$ be a sequence which converges weakly to
     $u\in H^1_{g,G}(M)$ and there exists $\varepsilon_0>0$ such that $0<\varepsilon_0\leq |\mathcal{F}(u_j)-\mathcal{F}(u)| \mbox{ for every }j \in \mathbb{N}.$
    As before,  $u_j \to u$ strongly in $L^p(M)$, $p\in (2,2^*)$. By  the mean value theorem one can see that for every $j\in \mathbb{N}$ there exists  $0<\theta_j<1$ such that
    $$0<\varepsilon_0\leq |\mathcal{F}(u_j)-\mathcal{F}(u)|\leq \int_M \alpha(x)|f(u+\theta_j(u_j-u))| |u_j-u|\d.$$ Now using assumptions ({$f_1$}) and
    ({$f_2$}),
   the right hand side of the above estimate tends to $0$,
     a contradiction. Thus, the energy functional $\mathcal{E}_{\lambda,G}$ is sequentially weakly lower semicontinuous.

\textsc{Claim 4.} (\textit{First solution}) By using assumptions
$(f_1)$ and $(f_2)$,  one has
$$\lim_{\mathscr H (u)\to 0}\frac{\mathcal{F}(u)}{\mathscr H
(u)}=\lim_{\mathscr H (u)\to \infty}\frac{\mathcal{F}(u)}{\mathscr H
(u)}= 0,$$ where $\ds
\mathscr{H}(u)=\frac{1}{2}\|u\|_{H^1_g(M)}^2+\frac{e}{4}\int_M
\phi_u u^2 \d$. Since $\alpha\in L^\infty(M)_+\setminus \{0\},$ on
account of
    ({$f_3$}), one can guarantee the existence of a suitable truncation
    function $u_{T}\in H^1_{g,G}(M)\setminus \{0\}$ such that
    $\mathcal{F}(u_T)>0.$ Therefore, we may define
    $$\ds\lambda_0=\inf_{\ds \substack{u\in H^1_{g,G}(M)\setminus\{0\} \\
            \mathcal{F}(u)>0}}\frac{\mathscr
            H(u)}{\mathcal{F}(u)}.$$
    The above limits imply that $0<\lambda_0<\infty.$ By \textsc{Claims} 1, 2 and 3,
    for every
    $\lambda>\lambda_0$, the functional $\mathcal{ E}_{\lambda,G}$ is
    bounded from below, coercive and satisfies the Palais-Smale condition. If we fix $\lambda>\lambda_0$ one can choose a function $w\in H^1_{g,G}(M)$
     such that $\mathcal{F}(w)>0$ and $\lambda>\frac{\mathscr H(w)}{\mathcal{F}(w)}\geq \lambda_0.$
     In particular,     $\ds c_1:=\inf _{H^{1}_{g,G}(M)}\mathcal E_{\lambda,G}\leq \mathcal E_{\lambda,G}(w)=\mathscr{H}(w)-\lambda \mathcal{F}(w)<0.$
    The latter inequality proves that the global minimum $u^1_{\lambda,G}\in H^{1}_{g,G}(M)$ of $\mathcal E_{\lambda,G}$ on $H^{1}_{g,G}(M)$ has
    negative energy level. In particular, $(u^1_{\lambda,G},\phi_{u^1_{\lambda,G}})\in H^{1}_{g,G}(M)\times H^{1}_{g,G}(M)$ is a  nontrivial weak solution to $(\mathcal{SM}_{\lambda})$.

\textsc{Claim 5.} (\textit{Second solution})  Let $q\in (2,2^*)$ be
fixed. By assumptions, for any $\eps>0$ there exists a constant
$C_\eps>0$ such that $$0\leq |f(s)|\leq
\frac{\eps}{\|\alpha\|_{L^\infty(M)}}|s|+C_\eps |s|^{q-1} \hbox{ for
all }s\in \mathbb{R}.$$ Then
\begin{align*}0\leq |\mathcal{F}(u)| &\leq \int_M
\alpha(x)|F(u(x))|\d \\ &\leq \int_M
\alpha(x)\left(\frac{\eps}{2\|\alpha\|_{L^\infty(M)}}u^2(x)+\frac{C_\eps}{q}|u(x)|^q\right)\d
\\ &\leq
\frac{\eps}{2}\|u\|^2_{H^1_g(M)}+\frac{C_\eps}{q}\|\alpha\|_{L^\infty(M)}
\kappa^q \|u\|^q_{H^1_g,(M)},\end{align*} where $\kappa_q$ is the
embedding constant in $H^1_{g,G}(M)\hookrightarrow L^q(M)$. Thus,
$$\mathcal{E}_{\lambda,G}(u)\geq
\frac{1}{2}(1-\lambda \eps)\|u\|^2_{H^1_g(M)}-\frac{\lambda
C_\eps}{q}\|\alpha\|_\infty \kappa_q^q \|u\|^q_{H^1_g(M)}.$$ Bearing
in mind that $q>2$, for enough small $\rho>0$ and $\eps
<\lambda^{-1}$ we have that
$$\ds
\inf_{\|u\|_{H^1_{g,G}(M)}=\rho}\mathcal{E}_{\lambda,G}(u)\geq
\frac{1}{2}\left(1-\eps \lambda\right) \rho-\frac{\lambda
C_\eps}{q}\|\alpha\|_{L^\infty(M)} \kappa_q^q \rho
^{\frac{q}{2}}>0.$$ A standard mountain pass argument (see
\cite{KVR,Willem}) implies the existence of a critical point
    $u^2_{\lambda,G}\in H^{1}_{g,G}(M)$ for $\mathcal E_{\lambda,G}$ with
    positive energy level. Thus
$(u^2_{\lambda,G},\phi_{u^2_{\lambda,G}})\in H^{1}_{g,G}(M)\times
H^{1}_{g,G}(M)$ is also a  nontrivial weak solution to
$(\mathcal{SM}_{\lambda})$. Clearly, $u^1_{\lambda,G}\neq
    u^2_{\lambda,G}$.
\hfill $\square$\\

\subsection{ Schr\"odinger-Maxwell systems involving oscillatory nonlinearities}

Before proving Theorem \ref{oszczero}, we need an auxiliary result.
Let us consider the system
\begin{equation*}
\ \left\{
\begin{array}{lll}
-\Delta_g u+u+e u\phi=\alpha(x)\widetilde f(u) & \mbox{in} &
M , \\
-\Delta_g \phi+\phi=qu^2 & \mbox{in} & M,
\end{array}%
\right. \eqno{(\widetilde{\mathcal{SM}}),}
\end{equation*}
where the following assumptions hold:
\begin{itemize}
    \item[$(\tilde f_1)$]\label{f1aux} $\widetilde f:[0,\infty)\to \mathbb{R}$ is a bounded function such that
    $f(0)=0$;
    \item[$(\tilde f_2)$]\label{f2aux} there are $0<a\leq b$ such that $\widetilde f(s)\leq 0$ for all $s\in [a,b]$.
\end{itemize}
Let  $x_0\in M$ be fixed, and $G\subset {\rm Isom}_g(M)$ and
$\alpha\in L^1(M)\cap L^\infty(M)$  be such that hypotheses
${\boldsymbol{\ds{(H_G^{x_0})}}}$ and
${\boldsymbol{\ds{(\alpha^{x_0})}}}$ are satisfied.

Let $\widetilde{\mathcal{E}}$ be the "one-variable" energy
functional associated with system $(\widetilde{\mathcal{SM}})$, and
 $\widetilde{ \mathcal{E}_G}$ be the restriction of
$\widetilde{\mathcal{E}}$ to the set $H^1_{g,G}(M)$. It is clear
that $\widetilde{\mathcal{E}}$ is well defined.
Consider the number $b\in \mathbb{R}$ from ({${\tilde f_2}$}); for
further use, we introduce the  sets $$W^b=\{u\in H^1_{g}(M):
\|u\|_{L^\infty(M)}\leq b\}\ \ \hbox{ and }\ \  W^b_G=W^b\cap
H^1_{g,G}(M).$$
\begin{proposition}\label{auxiliarylemma}  Let $(M,g)$ be an $n-$dimensional homogeneous Hadamard manifold $(3\leq n\leq 5)$, $x_0\in M$ be fixed, and $G\subset {\rm Isom}_g(M)$ and
    $\alpha\in  L^1(M)\cap
    L^\infty(M)$  be
    such that hypotheses   ${\boldsymbol{\ds{(H_G^{x_0})}}}$ and ${\boldsymbol{\ds{(\alpha^{x_0})}}}$ are satisfied. If $\widetilde{f}:[0,\infty)\to \R$ is a continuous function
    satisfying $(\tilde f_1)$ and $(\tilde f_2)$ then
    \begin{itemize}
        \item[\rm{(i)}] the infimum of $\widetilde{ \mathcal{E}_G}$ on $W^b_G$  is attained at an element $u_G \in W^b_G;$
        \item[\rm{(ii)}] $u_G(x)\in [0,a]$ a.e. $x \in M;$
        \item[\rm{(iii)}] $(u_G,\phi_{u_G})$ is a  weak solution to system $(\widetilde{\mathcal{SM}})$.
    \end{itemize}
\end{proposition}
{\it Proof.} \noindent   (i)
 By using the same method as in  \textsc{Claim 3} of the proof of Theorem \ref{theorem-sublinear}, the functional $\widetilde{ \mathcal{E}_G}$ is sequentially weakly lower semicontinuous on
 $H^1_{g,G}(M)$. Moreover, $\widetilde{ \mathcal{E}_G}$ is bounded from below.
The set $W^b_G$ is convex and closed in $H^1_{g,G}(M)$, thus weakly
closed. Therefore, the claim directly follows; let $u_G \in W^b_G$
be the infimum of $\widetilde{ \mathcal{E}_G}$ on $W^b_G$.

    (ii) Let $A=\{x\in M: u_G(x)\notin[0,a]\}$ and suppose
that the Riemannian measure of $A$ is positive. We consider  the
function $\gamma(s)=\min(s_+,a)$ and set $w=\gamma \circ u_G$. Since
$\gamma$ is Lipschitz continuous, then
    $w\in H^1_{g}(M)$ (see Hebey, \cite[Proposition 2.5, page 24]{hebey}). We claim that $w \in H^1_{g,G}(M)$.
    Indeed, for every $x\in M$ and $\sigma\in G$,
      $$\sigma w(x)=w(\sigma^{-1}(x))=(\gamma \circ u_G)(\sigma^{-1}(x))=\gamma(u_G(\sigma^{-1}(x)))=\gamma(u_G(x))=w(x).$$
      By construction, we clearly have that  $w \in W^b_G$. Let $$A_1=\{x\in A: u_G(x)<0\} \hbox{ and } A_2=\{x\in A: u_G(x)>a\}.$$
    Thus $A=A_1\cup A_2$, and from the construction we have that $w(x)=u_G(x)$ for all $x\in M\setminus A$, $w(x)=0$ for all $x\in A_1$, and $w(x)=a$ for all $x\in A_2$.  Now we have that
    \begin{align*}\widetilde{ \mathcal{E}_G}(w)-\widetilde{ \mathcal{E}_G}(u_G)=&  -\frac{1}{2}\int_A |\nabla_g u_G|^2dv_g+\frac{1}{2}\int_A (w^2-u_G^2)\d+\frac{e}{4}\int_A (\phi_w w^2-\phi_{u_G}u_G^2)\ \d \\ &- \int_A \alpha(x)\left(\widetilde F(w)-\widetilde F(u_G)\right)\ \d.
    \end{align*}
    Note that $$\int_A \left(w^2-u_G^2\right)\d=-\int_{A_1}u_G^2\d+\int_{A_2}\left(a^2-u_G^2\right)\d\leq 0.$$
    It is also clear that $\ds\int_{A_1}\alpha(x)(\widetilde F(w)-\widetilde F(u_G))\d=0,$
    and due to the mean value theorem and ({${\tilde f_2}$}) we have that $\ds\int_{A_2}\alpha(x)(\widetilde F(w)-\widetilde F(u_G))\d\geq 0.$ Furthermore,  $$\int_A (\phi_w w^2-\phi_{u_G}u_G^2) \d =-\int_{A_1} \phi_{u_G}u_G^2\d+\int_{A_2} (\phi_w w^2-\phi_{u_G}u_G^2) \d,$$
    thus due to Lemma \ref{comparison} (ii), since $0\leq w\leq u_G$, we have that $\ds\int_{A_2}(\phi_w w^2-\phi_{u_G}u_G^2)\d \leq 0.$ Combining the above estimates,
    we have  $\ds \widetilde{ \mathcal{E}_G}(w)-\widetilde{ \mathcal{E}_G}(u_G)\leq 0.$

    On the other hand, since $w\in W_G^b$ then $\ds \widetilde{ \mathcal{E}_G}(w)\geq \widetilde{ \mathcal{E}_G}(u_G)=\inf_{W_G^b}\widetilde{ \mathcal{E}_G}$, thus we necessarily have that
    $$\int_{A_1}u_G^2 \d=\int_{A_2}(a^2-u_G^2)\d=0,$$
    which implies that the Riemannian measure of $A$ should be zero, a contradiction. \vspace{0.2cm}\\
    (iii) The proof is divided into two steps:

    \textsc{Claim 1.} \textit{ $\widetilde{ \mathcal{E}}'(u_G)(w-u_G)\geq 0 \hbox{ for all }w\in W^b$.} It is clear that the set $W^b$ is closed and convex in $H_g^1(M)$.
    Let $\chi_{W^b}$ be the indicator function of the set $W^b$, i.e., $\kf(u)=0$ if $u \in W^b$, and $\kf(u)=+\infty$ otherwise.
    Let us consider the Szulkin-type functional $\mathscr{K}:H_g^1(M)\to \mathbb R\cup \{+\infty\}$ given by $\mathscr{K}=\widetilde{ \mathcal{E}}+\kf$. On account of the definition of the set $W^b_G$,
    the restriction of $\kf$ to $H^1_{g,G}(M)$ is precisely the indicator function $\kfG$ of the set $W^b_{G}$. By (i), since
    $u_G$ is a local minimum point of $\widetilde{ \mathcal{E}_G}$ relative to the set $W^b_G$, it is also a local minimum point of the Szulkin-type functional
    $\mathscr{K}_G=\widetilde{ \mathcal{E}_G}+\kfG$ on $H^1_{g,G}(M)$. In particular,  $u_G$ is a critical point of $\mathscr{K}_G$ in the sense of Szulkin \cite{Szulkin}, i.e.,
    $$0\in \widetilde{ \mathcal{E}_G}'(u_G)+\partial \kfG(u_G) \hbox{ in } \left(H^1_{g,G}(M)\right)^\star,$$
    where $\partial$ stands for the subdifferential in the sense of convex
    analysis.
    By exploring the compactness of the group $G$, we may apply
    the principle of symmetric criticality for Szulkin-type functionals, see Kobayashi and \^Otani \cite[Theorem
    3.16]{KO}, obtaining that
     $$0\in \widetilde{ \mathcal{E}}'(u_G)+\partial \kf(u_G) \hbox{ in } \left(H^1_{g}(M)\right)^\star.$$
    Consequently,   we have for every  $w\in W^b$ that $$0\leq \widetilde{ \mathcal{E}}'(u_G)(w-u_G)+\kf(w)-\kf(u_G),$$ which proves the claim.

    \textsc{Claim 2.} \textit{$(u_G,\phi_{u_G})$ is a  weak solution to the system $(\widetilde{\mathcal{SM}})$.} By assumption $(\tilde f_1)$ it is clear that $\ds
    C_\mathtt{m}=\sup_{s \in\mathbb{R}}|\widetilde f(s)|<\infty.$ The previous step and \eqref{derivalt} imply that for all $w\in W^b$,
\begin{eqnarray*}
  0 &\leq& \int_M \langle\nabla_g u_G ,\nabla_g(w-u_G)\rangle\d+\int_M u_G(w-u_G)\d \\
   &&+e\int_M u_G\phi_{u_G} (w-u_G)\d-\int_M \alpha(x)\widetilde
   f(u_G)(w-u_G)\d.
\end{eqnarray*}
Let us define the following function $$\zeta(s)=\left\{
    \begin{array}{ll}
    -b, & s<-b, \\
    s, & -b\leq s<b, \\
    b, & b\leq s.
    \end{array}
    \right.$$ Since $\zeta$ is Lipschitz continuous and $\zeta(0)=0$, then for fixed $\eps>0$ and $v\in H^1_g(M)$ the function
     $w_\zeta=\zeta\circ (u_G+\eps v)$ belongs to $H^1_g(M)$, see Hebey \cite[Proposition 2.5, page 24]{hebey}.
    By construction,  $w_\zeta \in W^b$.

    Let us denote by $B_1=\{x\in M:u_G+\eps v<-b\}$, $B_2=\{x\in M:-b\leq u_G+\eps v<b\}$ and
    $B_3=\{x\in M:u_G+\eps v\geq b\}$. Choosing $w=w_\zeta$ in the above inequality we have that $$0\leq I_1+I_2+I_3+I_4,$$
    where $$I_1=-\int_{B_1}|\nabla_g u_G|^2 \d+\eps \int_{B_2}\langle \nabla_g u_G,\nabla_g v\rangle \d-\int_{B_3}|\nabla_g u_G|^2\d,$$
     $$I_2=-\int_{B_1}u_G(b+u_G)\d+\eps \int_{B_2}u_Gv\d+\int_{B_3}(b-u_G)\d,$$
    $$I_3=-e\int_{B_1}u_G\phi_{u_G}(b+u_G)\d+\eps e\int_{B_2}u_G\phi_{u_G} v \d+e\int_{B_3}u_G\phi_{u_G}(b-u_G)\d,$$
    and  $$I_4=-\int_{B_1}\alpha(x)\widetilde f(u_G)(-b-u_G)\d-\eps \int_{B_2}\alpha(x)\widetilde f(u_G)v \d-\int_{B_3} \alpha(x) \widetilde f(u_G)(b-u_G)\d.$$
    After a rearrangement we obtain that
    \begin{align*}
    I_1+I_2+I_3+I_4 = & \eps \int_M \langle\nabla_g u_G,\nabla_g v \rangle\d+\eps \int_M u_Gv \d+\eps e\int_M u_G\phi_{u_G}v \d-\eps \int_M \alpha(x)f(u_G)v \d \\ &
    - \eps \int_{B_1}\langle\nabla_g u_G,\nabla_g v \rangle \d-\eps \int_{B_3}\langle \nabla_g u_G,\nabla_g v\rangle\d -\int_{B_1}|\nabla_g u_G|^2\d \\ &
    -\int_{B_3}|\nabla_g u_G|^2\d+\int_{B_1}(b+u_G+\eps v)\left(\alpha(x)\widetilde f(u_G)-u_G-eu_G\phi_{u_G}\right)\d \\ &+ \int_{B_3}(-b+u_G+\eps v)\left(\alpha(x)\widetilde f(u_G)-u_G-eu_G\phi_{u_G}\right)\d.
    \end{align*}
   Note that $$\int_{B_1}(b+u_G+\eps v)\left(\alpha(x)\widetilde f(u_G)-u_G-eu_G\phi_{u_G}\right)\d \leq -\eps \int_{B_1}\left(C_\mathtt{m} \alpha(x)+u_G+eu_G\phi_{u_G}\right)v \d,$$
    and
    $$\int_{B_3}(-b+u_G+\eps v)\left(\alpha(x)\widetilde f(u_G)-u_G-eu_G\phi_{u_G}\right)\d\leq \eps C_\mathtt{m} \int_{B_3}\alpha(x)v \d.$$
    Now, using the above estimates and dividing by $\eps>0$, we have that \begin{align*}0&\leq \int_M \langle \nabla_g u_G, \nabla_g v \rangle \d+\int_M u_Gv \d+e\int_M u_G\phi_{u_G}v \d-
    \int_M \alpha(x)\widetilde f(u_G)v \d \\ &-\int_{B_1}\left(\langle\nabla_g u_G,\nabla_g v\rangle+C_\mathtt{m} \alpha(x)v+u_Gv+eu_G\phi_{u_G}v\right)\d -\int_{B_3} \left(\langle\nabla_g u_G,\nabla_g v \rangle-
    C_\mathtt{m}\alpha(x)v\right)\d.
    \end{align*}
    Taking into account that the Riemannian measures for both sets  $B_1$ and $B_3$ tend to zero as   $\eps \to 0$, we get that $$0\leq  \int_M \langle\nabla_g u_G,\nabla_g v\rangle \d+\int_M u_Gv \d+e\int_M u_G\phi_{u_G}v \d-\int_M \alpha(x)\widetilde f(u_G)v \d.$$
    Replacing $v$ by  $(-v)$, it yields $$0=\int_M \langle\nabla_g u_G,\nabla_g v\rangle \d+\int_M u_Gv \d+e\int_M u_G\phi_{u_G}v \d-\int_M \alpha(x)\widetilde f(u_G)v \d,$$
    i.e., $\widetilde{ \mathcal{E}}'(u_G)=0.$ Thus $(u_G,\phi_{u_G})$ is a  $G-$invariant weak solution to $(\widetilde{\mathcal{SM}})$.
\hfill $\square$\\

Let $s>0$,  $0< r< \rho$ and
$A_{x_0}[r,\rho]=B_g(x_0,\rho+r)\setminus B_g(x_0,\rho-r)$ be an
annulus-type domain. For further use, we define the function
$w_s:M\to \mathbb{R}$ by $$w_s(x)=\left\{
    \begin{array}{ll}
    0, & x\in M\setminus A_{x_0}[r,\rho], \\
    s, & x\in A_{x_0}[r/2,\rho], \\
    \frac{2s}{r}(r-|d_g(x_0,x)-\rho|), & x\in A_{x_0}[r,\rho]\setminus
    A_{x_0}[r/2,\rho].
    \end{array}
    \right.$$
Note that  ${\boldsymbol{\ds{(H_G^{x_0})}}}$ implies  $w_s\in
H^1_{g,G}(M).$

\begin{proof}[Proof of Theorem \ref{oszczero}]
Due to ({$f_0^2$}) and the continuity of $f$ one can fix two
sequences $\{\theta_j\}_j, \{\eta_j\}_j$ such that $\ds
\lim_{j\to+\infty}\theta_j=\lim_{j\to+\infty}\eta_j=0,$ and for
every $j\in \N$, \begin{eqnarray}
    \label{kozrefogas} 0<\theta_{j+1}<\eta_j<s_j<\theta_j<1; \\\label{kozrefogas2} f(s)\leq 0 \hbox{ for every } s\in [\eta_j,\theta_j].
    \end{eqnarray}
   Let us introduce the auxiliary function  $f_j(s)=f(\min(s,\theta_j))$. Since  $f(0)=0$ (by ({$f_0^1$}) and ({$f_0^2$})), then  $f_j(0)=0$ and we may extend continuously
   the function $f_j$ to the whole real line by $f_j(s)=0$ if $s\leq 0$. For every $s\in \R$ and $j\in \N$,  we define $\ds F_j(s)=\int_0^s f_j(t){ \rm d}t.$
    It is clear that $f_j$ satisfies the assumptions ({$\tilde f_1$}) and ({$\tilde f_2$}). Thus, applying Proposition \ref{auxiliarylemma} to the function $f_j$,  $j\in\N$, the system \begin{equation}\label{oszcilaciohoz}
    \ \left\{
    \begin{array}{lll}
    -\Delta_g u+u+e u\phi=\alpha(x)f_j(u) & \mbox{in} &
    M , \\
    -\Delta_g \phi+\phi=qu^2 & \mbox{in} & M,
    \end{array}%
    \right.
    \end{equation} has a $G-$invariant weak solution $(u_j^0,\phi_{u_j^0})\in H^1_{g,G}(M)\times H^1_{g,G}(M)$ such that
    \begin{eqnarray}
    \label{elsokov} u_j^0 \in [0,\eta_j] \hbox{ a.e. } x\in M; \\
    \label{masodikkov} u_j^0 \hbox{ is the infimum of the functional }\mathcal{E}_j \hbox{ on the
    set}\
    W^{\theta_j}_G,
    \end{eqnarray}
    where $$\mathcal{E}_j(u)=\frac{1}{2}\|u\|_{H^1_g(M)}^2+\frac{e}{4}\int_M \phi_u u^2 \d-\int_M \alpha(x)F_j(u) \d.$$
By (\ref{elsokov}), $(u_j^0,\phi_{u_j^0})\in H^1_{g,G}(M)\times
H^1_{g,G}(M)$ is also a weak solution to the initial system
    $(\mathcal{SM})$.

    It remains to prove the existence of infinitely many distinct elements in the sequence $\{(u_j^0,\phi_{u_j^0})\}_j$.
   First, due to ${\boldsymbol{\ds{(\alpha^{x_0})}}}$,  there exist $0< r<
\rho$ such that ${\rm essinf}_{A_{x_0}[r,\rho]} \alpha>0.$ For
simplicity, let $D=A_{x_0}[r,\rho]$ and $K=A_{x_0}[r/2,\rho]$. By
({$f_0^1$}) there exist  $l_0>0$ and $\delta\in(0,\theta_1)$ such
that
\begin{equation}\label{becslesFrol}
    F(s)\geq -l_0 s^2 \hbox{ for every }s\in (0,\delta).
    \end{equation}
Again, ({$f_0^1$}) implies the existence of a non-increasing
sequence $\{\widetilde{s}_j\}_j\subset (0,\delta)$ such that
$\widetilde{s}_j\leq \eta_j$ and
\begin{equation}\label{becslesFrol-2}
F(\widetilde{s}_j)>L_0\widetilde{s}_j^2 \hbox{ for all }j\in \N,
    \end{equation}
where $L_0>0$ is enough large, e.g.,
\begin{equation}\label{L-null}
    L_0{\rm essinf}_K
\alpha>\frac{1}{2}\left(1+\frac{4}{r^2}\right){\rm
Vol}_g(D)+\frac{e}{4}\|\phi_\delta\|_{L^1(D)}+l_0
\|\alpha\|_{L^1(M)}.
\end{equation}
Note that
    \begin{align*}
    \mathcal{E}_j(w_{\widetilde{s}_j})      =& \frac{1}{2}\|w_{\widetilde{s}_j}\|_{H^1_g(M)}^2+\frac{e}{4}I_j-J_j,
    \end{align*}
where $$I_j=\int_D \phi_{w_{\widetilde{s}_j}} w_{\widetilde{s}_j}^2
\d \hbox{ and }J_j=\int_D \alpha(x)F_j(w_{\widetilde{s}_j})\d.$$ By
Lemma \ref{comparison} (ii) we have
$$I_j\leq \widetilde{s}_j^2\|\phi_\delta\|_{L^1(D)},\ j\in \mathbb N.$$
Moreover, by (\ref{becslesFrol}) and (\ref{becslesFrol-2}) we have
that
$$J_j\geq L_0 \widetilde{s}_j^2{\rm essinf}_K \alpha-l_0
\widetilde{s}_j^2\|\alpha\|_{L^1(M)},\ j\in \mathbb N.$$
Therefore, \begin{align*} \mathcal{E}_j(w_{\widetilde{s}_j})&\leq
\widetilde{s}_j^2\left(\frac{1}{2}\left(1+\frac{4}{r^2}\right){\rm
Vol}_g(D)+\frac{e}{4}\|\phi_\delta\|_{L^1(D)}+l_0
\|\alpha\|_{L^1(M)}-L_0 {\rm essinf}_K \alpha\right).
\end{align*}
Thus, in one hand, by (\ref{L-null}) we have
\begin{equation}\label{neg-energia}
    \ds\mathcal{E}_j(u_j^0)=\inf_{W^{\theta_j}_G}\mathcal{E}_j\leq
\mathcal{E}_j(w_{\widetilde{s}_j})<0.
\end{equation}
 On the other hand, by
(\ref{kozrefogas}) and (\ref{elsokov}) we clearly have
$$\ds\mathcal{E}_j(u_j^0)\geq -\int_M \alpha(x)F_j(u_j^0)\d=-\int_M \alpha(x)F(u_j^0)\d\geq -\|\alpha\|_{L^1(M)}\max_{s\in [0,1]}|f(s)|\eta_j,\ j\in \mathbb N.$$
Combining the latter relations, it yields that $\ds\lim_{j\to
+\infty}\mathcal{E}_j(u_j^0)=0.$ Since
$\mathcal{E}_j(u_j^0)=\mathcal{E}_1(u_j^0) \hbox{ for all }j\in \N,$
   we obtain that the sequence $\{u_j^0\}_j$ contains infinitely many distinct elements. In particular, by (\ref{neg-energia}) we have that
    $\ds\frac{1}{2}\|u_j^0\|_{H^1_g(M)}^2\leq \|\alpha\|_{L^1(M)} \max_{s\in [0,1]}|f(s)|\eta_j,$ which implies that $\ds\lim_{j\to \infty}\|u_j^0\|_{H^1_g(M)}=0.$ Recalling (\ref{a-tulajdonsag}),
     we also have $\ds\lim_{j\to \infty}\|\phi_{u_j^0}\|_{H^1_g(M)}=0,$
    which concludes the proof.
\end{proof}
\begin{remark}\rm
        Using Proposition \ref{auxiliarylemma} (i) and $\ds \lim_{j\to \infty}\eta_j=0$, it follows that $\ds \lim_{j\to \infty}\|u_j^0\|_{L^\infty(M)}=0.$
\end{remark}

\section*{Acknowledgement}The authors were supported by the  grant of the Romanian National Authority for Scientific Research, "Symmetries in elliptic problems: Euclidean and
non-Euclidean techniques", CNCS-UEFISCDI, project no.
PN-II-ID-PCE-2011-3-0241. A. Krist\'aly is also supported by the
J\'anos Bolyai Research Scholarship of the Hungarian Academy of
Sciences.

\end{document}